\documentclass[a4paper,11pt,dvips]{article}

\usepackage[utf8]{inputenc} 
\usepackage[english]{babel}
\usepackage{amsfonts,amsmath,amsthm,amssymb}
\usepackage{a4}
\usepackage[all]{xy}

\title{\sc{Variants of a-T-menability for actions on non-commutative $L_{p}$-spaces}}

\theoremstyle{proclaim}
\newtheorem{thm}{Theorem}[section]
\newtheorem{df}[thm]{Definition}
\newtheorem{prp}[thm]{Proposition}
\newtheorem{lem}[thm]{Lemma}
\newtheorem{rem}[thm]{Remark}
\newtheorem{cor}[thm]{Corollary}
\newtheorem*{ack}{Acknowledgments}

\newcommand{\esp}{\textrm{ }}
\newcommand{\n}{\vert\vert}

\author{\sc{Baptiste OLIVIER}}

\begin{document}
\maketitle

\begin{abstract}
We investigate various characterizations of the Haagerup property $(H)$ for a second countable locally compact group $G$, in terms of orthogonal representations of $G$ on $L_{p}(\mathcal{M})$, the non-commutative $L_{p}$-space associated with the von Neumann algebra $\mathcal{M}$.\\
For a semi-finite von Neumann algebra $\mathcal{M}$, we introduce a variant of property $(H)$, namely property $H_{L_{p}(\mathcal{M})}$, defined in terms of orthogonal representations on $L_{p}(\mathcal{M})$ which have vanishing coefficients. We study the relationships between properties $(H)$ and $(H_{L_{p}(\mathcal{M})})$ for various von Neumann algebras $\mathcal{M}$.\\
We also characterize property $(H)$ in terms of strongly mixing actions on $L_{p}(\mathcal{M})$ for some finite von Neumann algebras $\mathcal{M}$.\\
We finally give constructions of proper actions of groups with the Haagerup property by affine isometries on $L_{p}(\mathcal{M})$ for some algebras $\mathcal{M}$, such as the hyperfinite ${\rm II}_{\infty}$ factor $\mathcal{B}(l_{2})\otimes R$. 
\end{abstract}

\section{Introduction}

The Haagerup property or a-$T$-menability is an important property of locally compact topological group with several applications in geometry and in theory of operator algebras (see \cite{Cherix2001}). \\

Recall that a second countable locally compact group $G$ has the Haagerup property $(H)$ (or is a-T-menable) if there exists a unitary representation $\pi:G\rightarrow\mathcal{U}(\mathcal{H})$ on a Hilbert space $\mathcal{H}$ which has vanishing coefficients (that is, $g\mapsto<\pi(g)\xi,\eta>$ belongs to $C_{0}(G)$ for all $\xi,\eta\in\mathcal{H}$) and which almost has invariant vectors (that is, there exists a sequence of unit vectors $(\xi_{n})_{n}$ in $\mathcal{H}$ such that $\n\pi(g)\xi_{n}-\xi_{n}\n\rightarrow0$ uniformly on compact subsets of $G$). For short, we will say that $G$ has $(H)$. As is well-known, $G$ has $(H)$ if and only if $G$ is a-$T$-menable, that is, $G$ admits a proper action by affine isometries on some Hilbert space $\mathcal{H}$.\\

In \cite{bader2007propertyTLp}, a variant of property $(T)$ relative to Banach spaces was studied. It is called property $(T_{B})$, and it is defined by the means of orthogonal representations on the Banach space $B$. It is natural to investigate variants of property $(H)$ for group actions on a Banach space $B$. In the whole paper, we say that $\pi :G\rightarrow O(B)$ is an orthogonal representation if $\pi$ is a homomorphism from $G$ to $O(B)$ such that the maps $g\mapsto\pi(g)x$ are continuous for every $x\in B$.

\begin{df}{\rm
Let $B$ be a Banach space, and $B^{*}$ its dual space. Let $\pi:G\rightarrow O(B)$ be an orthogonal representation. We say that the representation $\pi$ has vanishing coefficients if
$$\lim_{g\rightarrow\infty}<\pi(g)x,y>=0\textrm{ for all }x\in B,\esp y\in B^{*}.$$
}
\end{df}

\begin{df}{\rm
Let $B$ be a Banach space. We say that a group $G$ has property $(H_{B})$ if there exists a representation $\pi:G\rightarrow O(B)$ with vanishing coefficients and almost invariant vectors.}
\end{df}

The first aim of this paper is to investigate this property in the case where $B=L_{p}(\mathcal{M})$, the $L_{p}$-space associated with the von Neumann algebra $\mathcal{M}$. We will study the relationships between property $(H)$ and property $(H_{L_{p}(\mathcal{M})})$ for some semi-finite von Neumann algebras $\mathcal{M}$.

\begin{rem}\label{rem-defHLp}
{\rm (i) Property $(H_{L_{p}(\mathcal{M})})$ is inherited by closed subgroups.\\

(ii) By analogy with property $(T)$ and property $(H)$, property $(H_{L_{p}(\mathcal{M})})$ is a strong negation of property $(T_{L_{p}(\mathcal{M})})$ in the following sense : if a topological group $G$ admits a closed normal non-compact subgroup $H$ such that the pair $(G,H)$ has property $(T_{L_{p}(\mathcal{M})})$, then $G$ does not have property $(H_{L_{p}(\mathcal{M})})$.
}
\end{rem}

To our knowledge, property $(H_{L_{p}([0,1])})$ has not been studied in the literature. Our first main result gives a characterization of linear Lie groups which have this property (compare with Theorem 4.0.1 in \cite{Cherix2001}).\\
\begin{thm}\label{thm-Lie}
Let $G$ be a connected Lie group such that in its Levi decomposition $G=SR$, the semi-simple part $S$ has finite center. Let $1\leq p<\infty$, $p\neq2$. Then the following are equivalent :\\
(i) $G$ has property $(H_{L_{p}([0,1])})$ ;\\
(ii) $G$ has the Haagerup property $(H)$ ;\\
(iii) $G$ is locally isomorphic to a product $\prod_{i\in I}S_{i}\times M$, where $M$ is amenable, $I$ is finite, and for every $i\in I$, $S_{i}$ is a group $SO(n_{i},1)$ or $SU(m_{i},1)$ with $n_{i}\geq2$, $m_{i}\geq1$.
\end{thm}

\begin{rem}
{\rm
(i) The previous theorem gives a classification of linear groups with property $(H_{L_{p}([0,1])})$ since, for any such group, the center of the semi-simple part in the Levi decomposition is finite (see Proposition 4.1 of Chapter ${\rm XVIII}$ in \cite{Hochschild1965structureLiegroups}).\\

(ii) We had to exclude groups $G=SR$ with $S$ having infinite center, as we are not be able to prove $(iii)\esp\Rightarrow (i)$ for the universal covers of $SO(n,1)$ and $SU(n,1)$.
}
\end{rem}
Theorem \ref{thm-Lie} has the following immediate consequence.
\begin{cor}\label{cor-Lie}
Let $G$ be a closed subgroup of a Lie group of the form $\prod_{i\in I}S_{i}\times M$, where $M$ is amenable, $I$ is finite, and for every $i\in I$, $S_{i}$ is a group $SO(n,1)$ or $SU(m,1)$. Then $G$ has property $(H_{L_{p}([0,1])})$ for all $1\leq p<\infty$.
\end{cor}

Let $1\leq p<\infty$, $p\neq2$. Depending on the von Neumann algebra $\mathcal{M}$ considered, properties $(H)$ and $(H_{L_{p}(\mathcal{M})})$ can be quite distinct. We show that, for $p\neq2$, property $(H_{C_{p}})$ is equivalent to property $(H)$ (Theorem \ref{thm-H=HCp}), and that only compact groups have property $(H_{S_{p}})$ (Theorem \ref{thm-HSp}).\\
Then we show that, among the connected second countable locally compact groups, only the compact ones have property $(H_{l_{p}})$.  Among the totally disconnected second countable groups, only the amenable ones have property $(H_{l_{p}})$ (see Theorem \ref{thm-lp}).\\

In \cite{Cherix2001} (see Theorem 2.1.5 and Proposition 2.3.1), Jolissaint studied property $(H)$ by the means of group actions by automorphisms of some von Neumann algebras. When dealing with finite von Neumann algebras, we can give the following definition of a strongly mixing representation.
\begin{df}\label{df-sm}{\rm
Let $\mathcal{M}$ be a finite von Neumann algebra with trace $\tau$. We say that a representation $\pi:G\rightarrow O(L_{p}(\mathcal{M}))$ is strongly mixing if
$$\lim_{g\rightarrow\infty}\tau(\pi(g)(x)y)=\tau(x)\tau(y)\textrm{ for all }x,y\in\mathcal{M}.$$
}
\end{df}

\begin{df}{\rm
Let $\mathcal{M}$ be a finite von Neumann algebra. We say that a group $G$ has property $(H_{L_{p}(\mathcal{M})}^{\rm mix})$ if there exists a representation $\pi:G\rightarrow O(L_{p}(\mathcal{M}))$ which is strongly mixing and has almost invariant vectors in the complement $L_{p}(\mathcal{M})'$ of the $\pi(G)$-invariant vectors.}
\end{df}

Here is our main result concerning the relationship between property $(H_{L_{p}(\mathcal{M})}^{mix})$ and property $(H)$.
\begin{thm}\label{thm-Hpmix}
Let $G$ be a locally compact second countable group. Let $1\leq p<\infty$, $p\neq2$. 
\begin{enumerate}
\item Let $\mathcal{M}$ be a finite von Neumann algebra, and let $1\leq p<\infty$. If $G$ has property $(H_{L_{p}(\mathcal{M})}^{mix})$, then $G$ has the Haagerup property $(H)$.
\item Assume that $G$ has the Haagerup property $(H)$. Then $G$ has property $(H_{L_{p}(\mathcal{M})}^{mix})$ in the two following cases.\\
(i) $(\mathcal{M},\tau)=(L^{\infty}([0,1]),\lambda)$ with $\lambda$ the Lebesgue measure ;\\
(ii) $\mathcal{M}=R$ is the hyperfinite ${\rm II}_{1}$ factor.
\end{enumerate}
\end{thm}

In \cite{Nowak2006} (see also \cite{ChatterjiCo2010medianspaces}), the author gave a characterization of a-$T$-menability with actions on commutative $L_{p}$-spaces. More precisely, he proves the following theorem :
\begin{thm}\label{thm-Nowak}{\rm \cite{Nowak2006}} Let $1<p<2$ and let $G$ be a second countable locally compact group. Then the following conditions are equivalent :\\
(i) $G$ has the Haagerup property $(H)$.\\
(ii) $G$ admits a proper affine isometric action on $L_{p}([0,1])$.
\end{thm}

Let $G$ be a second countable locally compact group with property $(H)$. By analogy with Theorem \ref{thm-Nowak}, we construct (see Theorem \ref{thm-proper1}) a proper action by affine isometries of $G$ on the space $L_{p}(l^{\infty}\otimes R)$. Using this construction, we obtain the following theorem.
\begin{thm}\label{thm-proper2}
Let $G$ be a second countable locally compact group with the Haagerup property. Let $1\leq p<\infty$. Then there exists a proper action of $G$ by affine isometries on $L_{p}(\mathcal{M})$, where $\mathcal{M}=\mathcal{B}(l_{2})\otimes R$ is the hyperfinite $II_{\infty}$ factor.
\end{thm}

This paper is organized as follows. In section 2, we recall some general facts about isometries of non-commutative $L_{p}(\mathcal{M})$-spaces, and orthogonal representations on $L_{p}(\mathcal{M})$. Section 3 is devoted to the study of property $(H_{L_{p}(\mathcal{M})})$ for the algebras $\mathcal{M}=l^{\infty}$, $\mathcal{M}=(\oplus_{n}\mathcal{M}_{n})_{\infty}$, $\mathcal{M}=\mathcal{B}(\mathcal{H})$, and $\mathcal{M}=L^{\infty}([0,1])$. In section 4, we show Theorem \ref{thm-Hpmix}. We give in section 5 two constructions of proper actions of a-$T$-menable groups by affine isometries on $L_{p}(\mathcal{M})$-spaces : one on the space $L_{p}(l^{\infty}\otimes R)$, and one on $L_{p}(\mathcal{B}(l_{2})\otimes R)$.

\begin{ack}{\rm
We wish to thank Bachir Bekka for all his very useful advice and comments on this paper. We are also grateful to the IRMAR for the stimulating atmosphere and the quality of working conditions. }
\end{ack}

\section{Orthogonal representations on non-commutative $L_{p}$-spaces}

Let $\mathcal{M}$ be a von Neumann algebra equipped with a semi-finite faithful normal trace $\tau$. The non-commutative $L_{p}(\mathcal{M},\tau)$-space is the completion of the set
\begin{displaymath}
\{x\in\mathcal{M}\esp\vert\esp\vert\vert x\vert\vert_{p}<\infty\esp\}
\end{displaymath}
with respect to the norm $\vert\vert x\vert\vert_{p}=\tau(\vert x\vert^{p})^{\frac{1}{p}}$. The dual space of $L_{p}(\mathcal{M},\tau)$ can be identified with $L_{p'}(\mathcal{M},\tau)$, where $p'$ is the conjugate exponent of $p$. In the case of $\tau(f)=\int f d\mu$ and $\mathcal{M}=L^{\infty}(X,\mu)$ for a measured space $(X,\mu)$, this is the classical space $L_{p}(X,\mu)$. For more details on non-commutative $L_{p}$-spaces, see the survey \cite{Pisier2003Lpspaces}.\\

Let $\mathcal{M}_{n}$ be the space of $n\times n$ matrices with complex coefficients. Now let $(\oplus_{n}\mathcal{M}_{n})_{\infty}=\{\oplus_{n}x_{n}\esp\vert\esp x_{n}\in\mathcal{M}_{n},\esp \sup_{n}\vert\vert x_{n}\vert\vert<\infty\}$, and $S_{p}=L_{p}(\mathcal{M})=\{\oplus_{n}x_{n}\esp\vert\esp x_{n}\in\mathcal{M}_{n},\sum_{n}{\rm Tr}(\vert x_{n}\vert^{p})<\infty \}$. Given a separable Hilbert space $\mathcal{H}$, denote by $C_{p}=\{x\in\mathcal{B}(\mathcal{H})\esp\vert\esp{\rm Tr} (\vert x\vert^{p})<\infty\}$ the Schatten $p$-ideals. We denote by $O(L_{p}(\mathcal{M}))$ the group of linear bijective isometries of $L_{p}(\mathcal{M})$. \\

The following result, due to F.J.Yeadon \cite{yeadon1980isometries}, gives a description of isometries of non-commutative $L_{p}$-spaces :
\begin{thm}\label{thm2.3}{\rm \cite{yeadon1980isometries}}
Let $1\leq p<\infty$ and $p\neq2$. Let $\mathcal{M}$ be a semi-finite von Neumann algebra with trace $\tau$. A linear map
$$U:L_{p}(\mathcal{M},\tau)\rightarrow L_{p}(\mathcal{M},\tau)$$
is a surjective isometry if and only if there exists :\\
- a normal Jordan *-isomorphism $J:\mathcal{M}\rightarrow\mathcal{M}$,\\
- a unitary $u\in\mathcal{M}$,\\
- a positive self-adjoint operator $B$ affiliated with $\mathcal{M}$ such that the spectral projections of $B$ commute with $\mathcal{M}$, the support of $B$ is $s(B)=1$, and $\tau(x)=\tau(B^{p}J(x))$ for all $x\in\mathcal{M}^{+}$,
satisfying 
$$U(x)=uBJ(x)\esp\textrm{for all } x\in\mathcal{M}\cap L_{p}(\mathcal{M}).$$
Moreover, such a decomposition is unique. We will say that such a decomposition is the Yeadon decomposition of the isometry $U$.
\end{thm}

We now give a more precise description of the isometries of $S_{p}=\{\oplus_{n}x_{n}\esp\vert\esp x_{n}\in\mathcal{M}_{n},\sum_{n}{\rm Tr}(\vert x_{n}\vert^{p})<\infty \}$.
\begin{lem}\label{lem2.4}
The two-sided ideals of $(\oplus_{n}\mathcal{M}_{n})_{\infty}=\{\oplus_{n}x_{n}\esp\vert\esp x_{n}\in\mathcal{M}_{n},\esp \sup_{n}\vert\vert x_{n}\vert\vert<\infty\}$ are the subspaces $\oplus_{i\in I}\mathcal{M}_{i}$ for $I\subset\mathbb{N}^{*}$.
\end{lem}

\begin{proof}
Let $I\subset\mathbb{N}^{*}$. It is clear that $\oplus_{i\in I}\mathcal{M}_{i}$ is a two-sided ideal of $\mathcal{M}$.\\

If $\mathcal{A}$ is a nonull two-sided ideal of $\mathcal{M}$, let $I\subset\mathbb{N}^{*}$ be a minimal set such that $\mathcal{A}\subset\oplus_{i\in I}\mathcal{M}_{i}$ and let $i\in I$. Then $\mathcal{M}_{i}\cap\mathcal{A}$ is a non-zero two-sided ideal of $\mathcal{M}_{i}$, so $\mathcal{M}_{i}\cap\mathcal{A}=\mathcal{M}_{i}$. Thus $\oplus_{i\in I}\mathcal{M}_{i}\subset\mathcal{A}$.
\end{proof}

\begin{lem}\label{lem2.5}
If $\mathcal{N}$ is a von Neumann algebra, $J$ a Jordan homomorphism on $\mathcal{N}$, and $\mathcal{A}$ a two-sided ideal of $\mathcal{N}$, then $J(\mathcal{A})$ is a two-sided ideal in $J(\mathcal{N})$.
\end{lem}

\begin{proof}
Recall from \cite{stormer1965jordanmorphisms} that $J=J^{1}+J^{2}$ with $J^{1}$ an algebra homomorphism and $J^{2}$ an algebra anti-homomorphism. More precisely, there exists two central projections $P_{1}, P_{2}\in\mathcal{N}$ such that $J^{1}(x)=J(P_{1}x)$ and $J^{2}(x)=J(P_{2}x)$ for all $x\in\mathcal{N}$. We also have $P_{1}P_{2}=0$. This implies that $J^{1}(x)J^{2}(y)=0$ for all $x, y\in\mathcal{N}$. It follows that, for $J(a)\in J(\mathcal{A})$ and $J(b), J(c)\in J(\mathcal{N})$, we have
\begin{displaymath}
\begin{split}
J(b)J(a)J(c)&=J(b)J^{1}(a)J(c)+J(b)J^{2}(a)J(c)\\
&=J^{1}(b)J^{1}(a)J^{1}(c)+J^{2}(b)J^{2}(a)J^{2}(c)\\
&=J^{1}(bac)+J^{2}(cab)\\
&=J(P_{1}bacP_{1}+P_{2}cabP_{2})\in J(\mathcal{A}).
\end{split}
\end{displaymath}
Hence $J(\mathcal{A})$ is a two-sided ideal of $J(\mathcal{N})$.
\end{proof}

\begin{prp}\label{prp2.6}
Let $U\in O(S_{p})$. There exist bijective isometries $U_{n}$ of $\mathcal{M}_{n}$ such that $U=\oplus_{n}U_{n}$. More precisely, there exist sequences $(u_{n})$, $(v_{n})$ of unitaries in $\mathcal{M}_{n}$ such that, for every $n$, we have
$$U_{n}(x)=u_{n}xv_{n}\textrm{ or }U_{n}(x)=u_{n}(^{t}x) v_{n}\textrm{ for all  }x\in S_{p}.$$
\end{prp}

\begin{proof}
Set $\mathcal{M}=(\oplus_{n}\mathcal{M}_{n})_{\infty}$. Let $U\in O(S_{p})$. By Theorem \ref{thm2.3}, we know that $U$ admits a Yeadon decomposition $U(x)=uBJ(x)$ for all $x\in \mathcal{M}\cap L_{p}(\mathcal{M})$.\\

Since $u\in\mathcal{M}$, we have $u=\oplus_{n}{u_{n}}$ for some $u_{n}\in\mathcal{M}_{n}$ for all $n$. Let $v$ be the inverse of $u$. Then $v\in\mathcal{M}$ since $v=u^{-1}=u^{*}$. Moreover, $v=\oplus_{n}v_{n}$, with $v_{n}=u_{n}^{*}$, and the relation $uv=1$ implies that $u_{n}v_{n}=1$ for all $n$. It follows that the $u_{n}$'s are unitaries. \\

Since $B$ commutes with $\mathcal{M}$ and since the support of $B$ is $s(B)=1$, there exist non-zero $(\lambda_{n})_{n}$ such that $B=\sum_{n}\lambda_{n}1$.\\

By Lemma \ref{lem2.5}, the Jordan isomorphism $J$ preserves the two-sided ideals. Then, for $n\geq 1$, $J(\mathcal{M}_{n})$ is an ideal of $\mathcal{M}$, and $J(\mathcal{M}_{n})=\oplus_{i\in I}\mathcal{M}_{i}$ for $I\subset\mathbb{N}^{*}$ by Lemma \ref{lem2.4}. If $i\in I$, then $i\leq n$ for dimension reasons, and $J^{-1}(\mathcal{M}_{i})$ is an ideal of $\mathcal{M}_{n}$, so $J^{-1}(\mathcal{M}_{i})=\{0\}$ or $J^{-1}(\mathcal{M}_{i})=\mathcal{M}_{n}$. Then $J(\mathcal{M}_{n})=\mathcal{M}_{n}$, $U(x)=\oplus_{n}\lambda_{n}u_{n}J(x_{n})$ for all $x=(x_{n})_{n}\in\mathcal{M}$. It is well known that isometries of $O(\mathcal{M}_{n})$ are of the form given in \ref{prp2.6}.
\end{proof}

\begin{prp}\label{prp2.7}
We have the following decomposition $O(L_{p}\oplus^{p} S_{p})=O(L_{p})\oplus O(S_{p}) $.
\end{prp}

\begin{proof}
By Theorem \ref{thm2.3}, it suffices to prove such a decomposition on a Jordan isomorphism $J$ of the von Neumann algebra $\mathcal{N}=L^{\infty}\oplus(\oplus_{n}\mathcal{M}_{n})_{\infty}$.\\

Recall that a projection $P$ in a von Neumann algebra $\mathcal{N}$ is said to be minimal if there is no projection $Q$ in $\mathcal{N}$ such that $0<Q<P$. Since a Jordan morphism preserves the projections and the order on the set of projections, the isomorphism $J$ preserves the minimal projections.\\

Clearly, the minimal projections of $L^{\infty}\oplus(\oplus_{n}\mathcal{M}_{n})_{\infty}$ are the rank one projections in $(\oplus_{n}\mathcal{M}_{n})_{\infty}$ and they generate the algebra $(\oplus_{n}\mathcal{M}_{n})_{\infty}$. Then we have $J((\oplus_{n}\mathcal{M}_{n})_{\infty})\subset(\oplus_{n}\mathcal{M}_{n})_{\infty}$, and the same argument for $J^{-1}$ gives the equality $J((\oplus_{n}\mathcal{M}_{n})_{\infty})=(\oplus_{n}\mathcal{M}_{n})_{\infty}$. Since $J$ is an isomorphism of $\mathcal{N}$, we have also $J(L^{\infty})=L^{\infty}$.
\end{proof}

Let $1\leq p,q<\infty$. Let $\mathcal{M}$ be a von Neumann algebra. Now we recall briefly how to get a representation on the non-commutative space $L_{q}(\mathcal{M})$ from a representation on $L_{p}(\mathcal{M})$. For an operator $a\in L_{p}(\mathcal{M})$, let $\alpha\vert a\vert$ be its polar decomposition. The map
\begin{displaymath}
\begin{split}
M_{p,q}:&L_{p}(\mathcal{M})\rightarrow L_{q}(\mathcal{M})\\
               &x=\alpha\vert a\vert\mapsto\alpha\vert a\vert^{\frac{p}{q}}
\end{split}
\end{displaymath}
is called the Mazur map.

\begin{thm}\label{thm2.9}
Let $G$ be a topological group. Let $1\leq p,q<\infty$ and $\mathcal{M}$ a semi-finite von Neumann algebra. Let $\pi^{p}$ be an orthogonal representation of a group $G$ on $L_{p}(\mathcal{M})$ such that :
$$\pi^{p}(g)(x)=u_{g}B_{g}J_{g}(x)\textrm{ for all }x\in L_{p}(\mathcal{M})\textrm{ and all }g\in G,$$
where $u_{g}$, $B_{g}$, and $J_{g}$ are the elements of the Yeadon decomposition of $\pi(g)$.\\
(i) Then the maps $\pi^{q}(g)=M_{p,q}\circ\pi^{p}(g)\circ M_{q,p}$ define an orthogonal representation $\pi^{q}$ of $G$ on $L_{q}(\mathcal{M})$ and we have the following decomposition for $\pi^{q}$ :
$$\pi^{q}(g)(x)=u_{g}B_{g}^{\frac{p}{q}}J_{g}(x)\textrm{ for all }x\in L_{q}(\mathcal{M})\textrm{ and all }g\in G.$$
(ii) The following relations hold for all $g_{1}, g_{2}\in G$ and all $x\in L_{p}(\mathcal{M})$,
\begin{displaymath}
\begin{split}
&u_{g_{1}g_{2}}=u_{g_{1}}J_{g_{1}}(u_{g_{2}}),\\
&B_{g_{1}g_{2}}=B_{g_{1}}J_{g_{1}}(B_{g_{2}}),\\
&J_{g_{1}g_{2}}(x)=J_{g_{1}}^{1}(J_{g_{2}}(x))+J_{g_{1}}^{2}(u_{g_{2}}J_{g_{2}}(x)u_{g_{2}}^{*}).
\end{split}
\end{displaymath}
\end{thm}

\begin{proof}
(i) The claim follows from a straightforward computation, using the structure of Jordan isomorphisms of von Neumann algebras (for details see Remark 3.3 in \cite{Olivier12TLp}).\\

(ii) Let $u\in\mathcal{M}$ be a unitary, $B$ a positive operator commuting with $\mathcal{M}$, and $y$ a positive element in $L_{p}(\mathcal{M})$. Let $J$ be a Jordan isomorphism of $\mathcal{M}$ and decompose $J=J^{1}+J^{2}$ with $J^{1}$ an algebra homomorphism and $J^{2}$ an algebra anti-homomorphism, as in Lemma \ref{lem2.5}. Then we have 
\begin{displaymath}
\begin{split}
J(uBy)&=J^{1}(uBy)+J^{2}(uyu^{*}Bu)\\
&=J^{1}(u)J^{1}(B)J^{1}(y)+J^{2}(u)J^{2}(B)J^{2}(uyu^{*})\\
&=J(u)J(B)(J^{1}(y)+J^{2}(uyu^{*})).
\end{split}
\end{displaymath}
Let $g_{1}, g_{2}\in G$ and $x$ is a positive element in $L_{p}(\mathcal{M})$. Then, using the computation above and the fact that the $B_{g}$'s commute with $\mathcal{M}$, we have
\begin{displaymath}
\begin{split}
\pi_{g_{1}}(\pi_{g_{2}}(x))&=u_{g_{1}}B_{g_{1}}J_{g_{1}}(u_{g_{2}}B_{g_{2}}J_{g_{2}}(x))\\
&=u_{g_{1}}J_{g_{1}}(u_{g_{2}})B_{g_{1}}J_{g_{1}}(B_{g_{2}})(J_{g_{2}}^{1}(x)+J_{g_{2}}^{2}(u_{g_{2}}xu_{g_{2}}^{*})).
\end{split}
\end{displaymath}
The claim follows from the unicity of the Yeadon decomposition in Theorem \ref{thm2.3}, and the morphism property for $\pi$. 
\end{proof}

Let $G$ be a topological group, $\mathcal{M}$ a von Neumann algebra, and $1\leq p<\infty$, $p\neq2$. Let $\pi^{p}:G\rightarrow O(L_{p}(\mathcal{M}))$ be an orthogonal representation of $G$. We denote by $$L_{p}(\mathcal{M})^{\pi^{p}(G)}=\{x\in L_{p}(\mathcal{M})\vert\esp\forall g\in G,\esp\pi^{p}(g)(x)=x\esp\}$$
the set of $\pi^{p}(G)$-invariant vectors. Let $p'$ be the conjugate exponent of $p$, and let $(\pi^{p}){'}$ be the contragradient representation of $\pi^{p}$ on $L_{p'}(\mathcal{M})$. 
We also denote by
$$L_{p}'=\{\esp x\in L_{p}(\mathcal{M})\esp\vert\esp\forall y\in L_{p'}^{(\pi^{p}){'}(G)},\esp\tau(xy)=0\esp\}$$
the $\pi(G)$-invariant complement of $L_{p}(\mathcal{M})^{\pi(G)}$ (see Proposition 2.6 in \cite{bader2007propertyTLp} for more details on this decomposition). The duality map $\ast:S(L_{p}(\mathcal{M}))\rightarrow S(L_{p'}(\mathcal{M}))$ between the unit spheres of $L_{p}(\mathcal{M})$ and $L_{p'}(\mathcal{M})$ is given by 
$$\ast x=M_{p,p'}(x)^{*}\textrm{ for all }x\in S(L_{p}).$$
Moreover, for all $g\in G$ and all $x\in S(L_{p'}(\mathcal{M}))$, we have 
$$(\pi^{p}){'}(g)x=\ast\circ\pi^{p}(g)\circ\ast^{-1}(x).$$
Hence we get the following description of the Yeadon decomposition of $(\pi^{p}){'}(g)$.
\begin{prp}\label{prp-contragradient}
Let $g\in G$. Let $\pi^{p}(g)=u_{g}B_{g}J_{g}$ be the Yeadon decomposition of $\pi^{p}(g)$. Then the Yeadon decomposition of $(\pi^{p}){'}(g)$ is given by
$$(\pi^{p}){'}(g)x=u_{g}^{*}B_{g}^{\frac{p}{p'}}u_{g}J_{g}(x)u_{g}^{*}\textrm{ for all }x\in L_{p}(\mathcal{M}).$$
\end{prp} 
The following proposition is an immediate consequence of the local uniform continuity of the Mazur map.
\begin{prp}\label{prp-conjugatepi}
If $\pi^{p}$ has almost invariant vectors in $L_{p}'$, then its conjugate by the Mazur map $\pi^{q}=M_{p,q}\circ\pi^{p}\circ M_{q,p}$ has almost invariant vectors in $L_{q}'$.
\end{prp}

\section{Representations on $L_{p}(\mathcal{M})$ with vanishing coefficients}

We investigate property $(H_{L_{p}(\mathcal{M})})$ and its relationships with property $(H)$ for the following von Neumann algebras : $\mathcal{M}=L^{\infty}([0,1])$, $\mathcal{M}=(\oplus_{n}\mathcal{M}_{n})_{\infty}$, $\mathcal{M}=\mathcal{B}(\mathcal{H})$, and $\mathcal{M}=l^{\infty}$.

\subsection{Property $(H_{L_{p}([0,1])})$}

In this subsection, we prove Theorem \ref{thm-Lie}. The most difficult part of this proof is the implication $(iii)\esp\Rightarrow (i)$. We mention that, even for the groups $SO(n,1)$ or $SU(n,1)$, the proof is not elementary since it depends heavily on the fact that these groups have lattices $\Gamma$ with non-trivial first Betti number, that is, lattices with infinite abelianization. The latter result was shown by Millson in \cite{Millson1976} for the case $SO(n,1)$, and by Kazhdan in \cite{Kazhdan1977} for the case $SU(n,1)$.\\

In the proof of $(iii)\esp\Rightarrow (i)$, we will need the following technical lemma which asserts that vanishing coefficients and almost invariant vectors are preserved for the quasi-regular representation, when passing from a group to a finite cover, and from the finite cover to the group.
\begin{lem}\label{lem-technicalLie}
Let $G_{1}$ and $G_{2}$ be locally compact topological groups and $p:G_{1}\rightarrow G_{2}$ be a finite covering.
\begin{enumerate}
\item Let $H_{2}$ be a closed subgroup of $G_{2}$ such that $G_{2}/H_{2}$ carries a $G_{2}$-invariant measure, and the quasi-regular representation $\lambda_{G_{2}/H_{2}}$ has almost invariant vectors and vanishing coefficients. Set $H_{1}=p^{-1}(H_{2})$. Then $\lambda_{G_{1}/H_{1}}$ has almost invariant vectors and vanishing coefficients.
\item Let $H_{1}$ be a closed subgroup of $G_{1}$ such that $G_{1}/H_{1}$ carries a $G_{1}$-invariant measure, and the quasi-regular representation $\lambda_{G_{1}/H_{1}}$ has almost invariant vectors and vanishing coefficients. Set $H_{2}=p(H_{1})$. Then $H_{2}$ is closed and $\lambda_{G_{2}/H_{2}}$ has almost invariant vectors and vanishing coefficients.
\end{enumerate}
\end{lem}

\begin{proof}
In the two cases, let $\overline{p}:G_{1}/H_{1}\rightarrow G_{2}/H_{2}$ be the map induced by the covering map $p:G_{1}\rightarrow G_{2}$. Observe that $\overline{p}$ is $G_{1}$-invariant, for the natural action of $G_{1}$ on $G_{1}/H_{1}$ and the action of $G_{1}$ on $G_{2}/H_{2}$ given by $p$ :
$$g_{1}.(g_{2}H_{2})=p(g_{1})g_{2}H_{2}\textrm{ for all }g_{1}\in G_{1}, g_{2}\in G_{2}.$$
Since $Z_{1}={\rm Ker}(p)$, the map $\overline{p}$ has finite fibers : indeed, the fiber over $p(g_{1}H_{2})$ is $\{g_{1}zH_{1}\esp\vert\esp z\in Z_{1}\esp\}=\{zg_{1}H_{1}\esp\vert\esp z\in Z_{1}\esp\}$, as $Z_{1}$ is central.

\begin{enumerate}
\item Let $\mu_{2}$ be a $G_{2}$-invariant measure on $G_{2}/H_{2}$. In this case, $Z_{1}\subset H_{1}$ and so $\overline{p}$ is bijective. Then $\mu_{1}=\overline{p}^{-1}\ast\mu_{2}$ is a $G_{1}$-invariant measure on $G_{1}/H_{1}$. The quasi-regular representation $\lambda_{G_{1}/H_{1}}$ is equivalent to the representation $\lambda_{G_{2}/H_{2}}\circ p$.\\
Since $1_{G_{2}}\prec\lambda_{G_{2}/H_{2}}$ and $\lambda_{G_{2}/H_{2}}$ has vanishing coefficients, we have $1_{G_{1}}\prec\lambda_{G_{1}/H_{1}}$ and $\lambda_{G_{1}/H_{1}}$ has vanishing coefficients.

\item Notice that $H_{2}=p(H_{1})$ is a closed subgroup of $G_{2}$ since the cover $p:G_{1}\rightarrow G_{2}$ is finite. Let $\mu_{1}$ be a $G_{1}$-invariant measure on $G_{1}/H_{1}$. Since the fibers of $\overline{p}$ are finite, we can define a $G_{2}$-invariant measure $\mu_{2}$ on $G_{2}/H_{2}$ by
$$\int_{G_{2}/H_{2}}f\esp d\mu_{2}=\int_{G_{1}/H_{1}}f\circ\overline{p}\esp d\mu_{1}\textrm{ for all }f\in C_{c}(G_{2}/H_{2}).\esp (\ast)$$
The induced mapping
\begin{displaymath}
\begin{split}
\psi:\esp&L_{2}(G_{2}/H_{2},\mu_{2})\rightarrow L_{2}(G_{1}/H_{1},\mu_{1})\\
&\esp f\esp\mapsto\esp f\circ\overline{p}
\end{split}
\end{displaymath}
is a linear isometry which intertwines the $G_{1}$-representations $\lambda_{G_{2}/H_{2}}\circ p$ and $\lambda_{G_{1}/H_{1}}$. So, $\lambda_{G_{2}/H_{2}}\circ p$ is equivalent to a subrepresentation of $\lambda_{G_{1}/H_{1}}$. Since $\lambda_{G_{1}/H_{1}}$ has vanishing coefficients, the same is true for $\lambda_{G_{2}/H_{2}}\circ p$. As $p$ is surjective and had finite kernel, it follows that the $G_{2}$-representation $\lambda_{G_{2}/H_{2}}$ has vanishing coefficients.\\

It remains to prove that $1_{G_{2}}\prec\lambda_{G_{2}/H_{2}}$. To show this
we claim that
$${\rm Im}(\psi)=L_{2}(G_{1}/H_{1})^{\lambda_{G_{1}/H_{1}}(Z_{1})},$$
the space of $Z_{1}$-invariant vectors in $L_{2}(G_{1}/H_{1})$. \\
Indeed, let $f_{1}\in L_{2}(G_{1}/H_{1})^{\lambda_{G_{1}/H_{1}}(Z_{1})}$. As mentioned above, the fiber over $p(g_{1}H_{2})$ is $\{ zg_{1}H_{1}, z\in Z_{1}\}$ for every $g_{1}\in H_{1}$. Hence, $f_{1}$ is constant on the fibers of $\overline{p}$ and there exists a map $f_{2}$ on $G_{2}/H_{2}$ such that $f_{2}\circ\overline{p}= f_{1}$. It is clear that $f_{2}\in L_{2}(G_{2}/H_{2})$ by formula $(\ast)$.\\
Conversely, if $f\in L_{2}(G_{2}/H_{2})$, it is clear that $f\circ \overline{p}$ is a $Z_{1}$-invariant function in $L_{2}(G_{1}/H_{1})$.\\

We now show that $\lambda_{G_{2}/H_{2}}$ almost has invariant vectors. It suffices to show that the restriction of $\lambda_{G_{1}/H_{1}}$ to the subspace $L_{2}(G_{1}/H_{1})^{\lambda_{G_{1}/H_{1}}(Z_{1})}$ almost has invariant vectors. Take a sequence $(v_{n})_{n}$ of almost invariant vectors for $\lambda_{G_{1}/H_{1}}$. For $n\in \mathbb{N}$, define
$$w_{n}=\frac{1}{\vert Z_{1}\vert}\sum_{z\in Z_{1}}\lambda_{G_{1}/H_{1}}(z)v_{n}.$$
For every $n\in\mathbb{N}$, $w_{n}\in L_{2}(G_{1}/H_{1})^{\lambda_{G_{1}/H_{1}}(Z_{1})}$. Moreover, for $g\in G$,
$$\n\lambda_{G_{1}/H_{1}}(g)w_{n}-w_{n} \n_{2}\leq\frac{1}{\vert Z_{1}\vert}\sum_{z\in Z_{1}}\n\lambda_{G_{1}/H_{1}}(zg)v_{n}-v_{n} \n_{2}$$
so that $\lim_{n}\sup_{g\in K}\n \lambda_{G_{1}/H_{1}}(g)w_{n}-w_{n}\n_{2}=0$ for every compact $K\subset G$.
We have 
$$\n w_{n}-v_{n}\n_{2}\leq \frac{1}{\vert Z_{1}\vert}\sum_{z\in Z_{1}}\n\lambda_{G_{1}/H_{1}}(z)v_{n}-v_{n} \n_{2}$$
and the left side of the inequality tends to $0$; hence, since $\n v_{n}\n_{2}=1$, $\lim_{n}\n w_{n}\n_{2}=1$. The sequence $(\tilde{w_{n}})_{n}$, defined by $\tilde{w_{n}}=\frac{1}{\n w_{n}\n_{2}}w_{n}$, is a sequence of almost invariant vectors for the restriction of $\lambda_{G_{1}/H_{1}}$ to\\
$L_{2}(G_{1}/H_{1})^{\lambda_{G_{1}/H_{1}}(Z_{1})}$. The lemma is proved.

\end{enumerate}
\end{proof}

We are now able to give the proof of Theorem \ref{thm-Lie}.
\begin{proof}[Proof of theorem \ref{thm-Lie}]

(i) $\Rightarrow$ (ii) : Assume that $G$ is a connected Lie group without property $(H)$. By Theorem 3.3.1 of Cornulier in \cite{Cornulier2006Kazhdanpair}, there exists a \emph{normal} subgroup $R_{T}$ in $G$ such that $G/R_{T}$ has the Haagerup property, and the pair $(G,R_{T})$ has property $(T)$. Since $G/R_{T}$ has the Haagerup property $(H)$, and $G$ does not have property $(H)$, the subgroup $R_{T}$ is non-compact. By Theorem A in \cite{bader2007propertyTLp}, the pair $(G,R_{T})$ has property $(T_{L_{p}([0,1])})$ for every $1<p<\infty$. Hence, by Remark \ref{rem-defHLp}, $G$ does not have property  $(H_{L_{p}([0,1])})$. \\

(ii) $\Rightarrow$ (iii) : This is the result from \cite{Cherix2001}, stated in Theorem 4.0.1.\\

(iii) $\Rightarrow$ (i) : We will show that $G$ admits a closed subgroup $H$ such that the quasi-regular representation $\lambda_{G/H}:G\rightarrow O(L_{2}(G/H))$ has almost invariant vectors and vanishing coefficients. Then we will conjugate this representation by the Mazur map to obtain the desired representation on $L_{p}(G/H)$.\\

Since the semi-simple part $S$ of $G$ has finite center, and since $G$ is locally isomorphic to the direct product $\prod_{i}S_{i}\times M$, using Proposition 8.1 in \cite{Cowling2005Lie}, there exists a finite covering $p:G^{\natural}\rightarrow G$ such that $G^{\natural}$ is a direct product of closed connected subgroups $\prod_{i}S_{i}^{\natural}\times M^{\natural}$, where every $S_{i}^{\natural}$ is a simple Lie group with finite center, and $M^{\natural}$ is amenable.\\

Let $i\in I$. Let $S_{i}=SO(n_{i},1)$ or $S_{i}=SU(m_{i},1)$ be such that $S_{i}$ is locally isomorphic to $S_{i}^{\natural}$. By the results in \cite{Millson1976} and \cite{Kazhdan1977}, there exists a lattice $\Gamma_{i}$ in $S_{i}$ such that $\vert\Gamma_{i}/[\Gamma_{i},\Gamma_{i}]\vert=\infty$.\\

Set $G_{1}=\prod_{i}S_{i}^{\natural}\times M^{\natural}$. We consider the closed subgroup of $G_{1}$ defined by 
$$H_{1}=\prod_{i\in I}[\Gamma_{i},\Gamma_{i}]\times \{e\}.$$ 
We claim that the quasi-regular representation $\lambda_{G_{1}/H_{1}}$ of $G_{1}$ on $L_{2}(G_{1}/H_{1})$ has almost invariant vectors and vanishing coefficients. Indeed, we have 
$$\lambda_{G_{1}/H_{1}}\simeq\otimes_{i\in I}\lambda_{S_{i}/[\Gamma_{i},\Gamma_{i}]}\otimes\lambda_{M^{\natural}},\esp (1)$$
the right hand-side being the exterior tensor product of the representations.\\

We first show that $\lambda_{G_{1}/H_{1}}$ has vanishing coefficients. Since the representation $\lambda_{M^{\natural}}$ has vanishing coefficients, it suffices to show that $\lambda_{S_{i}/[\Gamma_{i},\Gamma_{i}]}$ has vanishing coefficients for every $i\in I$. By the Howe-Moore theorem, this is the case if and only if
$$L_{2}(S_{i}/[\Gamma_{i},\Gamma_{i}])^{\lambda_{S_{i}/[\Gamma_{i},\Gamma_{i}]}(S_{1})}=\{0\}\textrm{ for all }i\in I.$$

To show this, let $i\in I$ be fixed. Since $\Gamma_{i}/[\Gamma_{i},\Gamma_{i}]$ is infinite, the space $l_{2}(\Gamma_{i}/[\Gamma_{i},\Gamma_{i}])$ does not have non-zero $\lambda_{\Gamma_{i}/[\Gamma_{i},\Gamma_{i}]}(\Gamma_{i})$-invariant vector. Since $S_{i}/\Gamma_{i}$ carries a $S_{i}$-invariant finite measure, this implies by induction (see theorem E.3.1 in \cite{bekka2008kazhdan}) that 
$$L_{2}(S_{i}/[\Gamma_{i},\Gamma_{i}])^{\lambda_{S_{i}/[\Gamma_{i},\Gamma_{i}
 ]}(S_{i})}=\{0\}$$
We have therefore proved that $\lambda_{G_{1}/H_{1}}$ has vanishing coefficients.\\

Next we show that $\lambda_{G_{1}/H_{1}}$ almost has invariant vectors. For this, it suffices to show that $\lambda_{M^{\natural}}$ and all $\lambda_{S_{i}/[\Gamma_{i},\Gamma_{i}
 ]}$ have almost invariant vectors, by formula $(1)$. Indeed, by the Hulanicki-Reiter Theorem, this is clear for $\lambda_{M^{\natural}}$ since $M^{\natural}$ is amenable. Fix $i\in I$. Since $\Gamma_{i}/[\Gamma_{i},\Gamma_{i}]$ is abelian and therefore amenable, we have $1_{\Gamma_{i}}\prec\lambda_{\Gamma_{i}/[\Gamma_{i},\Gamma_{i}]}$. Since $S_{i}/\Gamma_{i}$ has a finite $S_{i}$-invariant measure, we have also $1_{S_{i}}\prec\lambda_{S_{i}/[\Gamma_{i},\Gamma_{i}]}$. \\

Denote by $\overline{S_{i}}=PSO(n_{i},1)$ or $\overline{S_{i}}=PSU(m_{i},1)$ the quotient of $S_{i}$ by its (finite) center. Denote by $G_{2}$ and $G_{3}$ the groups 
$$G_{2}=\prod_{i\in I}\overline{S_{i}}\times M^{\natural}\textrm{ and }G_{3}=\prod_{i\in I}S_{i}^{\natural}\times M^{\natural}.$$
Observe that we have three finite covering maps $p_{1}:G_{1}\rightarrow G_{2}$, $p_{2}:G_{3}\rightarrow G_{2}$ and $p_{3}\esp(=p):G_{3}\rightarrow G$. We apply now Lemma \ref{lem-technicalLie} successively to $p_{1}$, $p_{2}$ and $p_{3}$. We obtain the existence of a closed subgroup $H$ in $G$ such that $G/H$ carries a $G$-invariant measure, and $\lambda_{G/H}$ almost has invariant vectors and has vanishing coefficients.\\

Let $\pi^{p}$ be the orthogonal representation of $G$ on $L_{p}(G/H)$ defined by the same formula as $\lambda_{G/H}$ :
$$\pi^{p}(g)f(g'H)=f(g^{-1}g'H)\textrm{ for all }f\in L_{p}(G/H)\textrm{ and }g,g'\in G.$$
By Proposition \ref{prp-conjugatepi}, $\pi^{p}$ almost has invariant vectors. Since $G/H$ carries a $G$-invariant measure, we have 
Moreover, for $x,y\in C_{c}(G/H)$, the matrix coefficient $g\mapsto<\pi^{p}(g)x,y>$ is in $C_{0}(G)$. By density of $C_{c}(G/H)$ in $L_{p}(G/H)$, the representation $\pi^{p}$ has vanishing coefficients.

\end{proof}

\subsection{Properties $(H_{S_{p}})$ and $(H_{C_{p}})$}

Now let us study $(H_{L_{p}(\mathcal{M})})$ for the two discrete von Neumann algebras $\mathcal{B}(\mathcal{H})$ and $(\oplus\mathcal{M}_{n})_{\infty}=\{ x=\oplus_{n} x_{n}\esp\vert\esp\sup_{n}\n x_{n}\n<\infty\esp\}$.

\begin{thm}\label{thm-H=HCp}
Let $1\leq p<\infty$, $p\neq2$. Let $G$ be a locally compact second countable group. Then the following properties are equivalent.\\
(i) $G$ has property $(H_{C_{p}})$.\\
(ii) $G$ has property $(H)$.
\end{thm}

\begin{proof}
(i) $\Rightarrow$ (ii) : Let $\pi^{p}:G\rightarrow O(C_{p})$ be an orthogonal representation with vanishing coefficients and which almost has invariant vectors. By Theorem \ref{thm2.9} and Arazy's description of $O(C_{p})$ (see \cite{Arazy1975isometriesCp}), every isometry $\pi^{2}(g)=M_{p,2}\circ \pi^{p}(g)\circ M_{2,p}$ coincides with $\pi^{p}(g)$ on the set of finite rank operators, hence the representation $\pi^{2}:G\rightarrow O(C_{2})$ has vanishing coefficients. By Proposition \ref{prp-conjugatepi}, $\pi^{2}$ almost has invariant vectors. Hence $G$ has property $(H)$.\\

(ii) $\Rightarrow$ (i) : Let $\rho:G\rightarrow\mathcal{U}(\mathcal{H})$ be a unitary representation of the group $G$ on the Hilbert space $\mathcal{H}$ with almost invariant vectors and vanishing coefficients.\\
Define
$$\pi^{p}(g)(x)=\rho(g)x\rho(g^{-1})\textrm{ for }x\in\mathcal{B}(\mathcal{H}).$$
Clearly, the previous formula defines an orthogonal representation $\pi^{p}:G\rightarrow O(C_{p})$. Let us show that $\pi^{p}$ has vanishing coefficients. \\
By density of finite rank operators and linearity of the trace, it suffices to show that $\lim_{g\rightarrow\infty}{\rm Tr}(\pi^{p}(g)(x)y)=0$ for $x,y$ positive finite rank operators. This is straightforward to check. Indeed, we can write
\begin{displaymath}
\begin{split}
&x=\sum_{i=1}^{n}<.,\xi_{i}>\xi_{i}\\
&y=\sum_{j=1}^{m}<.,\eta_{j}>\eta_{j}
\end{split}
\end{displaymath}
Let $(\zeta_{k})$ be an orthonormal basis of $\mathcal{H}$. Then, for vectors $\zeta_{i}, \eta_{j}\in\mathcal{H}$,
\begin{displaymath}
\begin{split}
{\rm Tr} (\pi^{p}(g)(x)y)&=\sum_{k}<\pi^{p}(g)(x)y\zeta_{k},\zeta_{k}>\\
&=\sum_{k}\sum_{j=1}^{m}\sum_{i=1}^{n}<\zeta_{k},\eta_{j}><\rho(g^{-1})\eta_{j},\xi_{i}><\rho(g)\xi_{i},\zeta_{k}>
\end{split}
\end{displaymath}
Hence,
\begin{displaymath}
\begin{split}
\vert{\rm Tr}(\pi^{p}(g)(x)y)\vert&\leq\sum_{j=1}^{m}\sum_{i=1}^{n}\vert<\rho(g^{-1})\eta_{j},\xi_{i}>\vert\vert\vert\eta_{j}\vert\vert\vert\vert\xi_{i}\vert\vert.
\end{split}
\end{displaymath}
As $\pi^{p}$ has vanishing coefficients, the right side of the inequality tends to $0$ as $g$ tends to infinity.\\

It remains to show that $\pi^{p}$ has almost invariant vectors in $C_{p}$. In view of Proposition \ref{prp-conjugatepi}, it suffices to prove this for $p=2$. For $\xi\in\mathcal{H}$ with $\n\xi\n=1$, denote by $P_{\xi}\in C_{2}$ the orthogonal projection on the subspace $\mathbb{C}\xi$. Observe that $\n P\xi\n_{2}=1$ and for $\xi, \eta$ two unit vectors in $\mathcal{H}$, we have
$$\n P_{\xi}-P_{\eta}\n_{2}\leq 2\n\xi-\eta\n.$$
Let $(\xi_{n})_{n}$ be a sequence of almost invariant vectors for $\rho$. Set $v_{n}=P_{\xi_{n}}$ for all $n$. Then, for every $g\in G$, $\pi^{2}(g)(v_{n})=P_{\rho(g)\xi_{n}}$. The previous inequality therefore shows that $(v_{n})_{n}$ is a sequence of almost invariant vectors for $\pi^{2}$.\\
Hence $\pi^{p}$ almost has invariant vectors and has vanishing coefficients, thus $G$ has property $(H_{C_{p}})$.
\\

\end{proof}

The following proposition implies that only compact groups have property $(H_{S_{p}})$ or property $(H_{L^{p}\oplus^{p} S_{p}})$ for $p\neq2$.
\begin{thm}\label{thm-HSp}
Let $G$ be a non-compact locally compact group. Let $p\neq2$. There is no representation of $G$ on $S_{p}$ or on $L^{p}\oplus^{p} S_{p}$. Consequently, $G$ does not have property $(H_{S_{p}})$ or property $(H_{L^{p}\oplus^{p} S_{p}})$.
\end{thm}

\begin{proof}
Assume, by contradiction, that we have a representation $\pi:G\rightarrow O(S_{p})$ with vanishing coefficients. By Proposition \ref{prp2.6}, such a representation can be written as a sum $\pi=\oplus_{n}\pi_{n}$ of representations $\pi_{n}$ on $\mathcal{M}_{n}$. For all $n$ and $g\in G$, there exist $u_{n}(g)$, $v_{n}(g)$ unitaries in $\mathcal{M}_{n}$ such that $\pi_{n}(g)x=u_{n}(g)xv_{n}(g)$ or $\pi_{n}(g)x=u_{n}(g)(^{t}x) v_{n}(g)$ for all $x\in\mathcal{M}_{n}$. Since $\pi$ has vanishing coefficients, each $\pi_{n}$ has also vanishing coefficients. Hence
$$\lim_{g\rightarrow\infty}{\rm Tr}_{n}(u_{n}(g)v_{n}(g)x)=0\textrm{ for all }x\in\mathcal{M}_{n}.$$
This implies that every coefficient of the matrix $u_{n}(g)v_{n}(g)$ tends to $0$ as $g$ tends to $\infty$, which contradicts the facts that $\vert\vert u_{n}(g)v_{n}(g)\vert\vert=1$ for all $g\in G$ and that $G$ is non-compact.\\

The claim about property $(H_{L^{p}\oplus^{p} S_{p}})$ is proved with a similar way, using the decomposition $O(L^{p}\oplus^{p} S_{p})=O(L^{p})\oplus O(S_{p}) $ from Proposition \ref{prp2.7}.

\end{proof}

\subsection{Property $(H_{l_{p}})$}

Let $G$ be a second countable locally compact group, and let $1\leq p<\infty$, $p\neq2$. Let $\pi^{p}:G\rightarrow O(l_{p})$ be a representation $G$ on $l_{p}$. \\

It was proved in \cite{BBBO2013Tlp} (Corollary 20) that the conjugate representation $\pi^{2}$ of $G$ on $l_{2}$ is unitarily equivalent to a sum of monomial representations associated to open subgroups of $G$, that is, there exist a family of open subgroups $(H_{i})_{i\in I}$ of $G$ and unitary characters $(\chi_{i})_{i\in I}$ on the $H_{i}$'s such that
$$\pi^{2}=M_{p,2}\circ \pi^{p}\circ M_{2,p}\simeq \oplus_{i\in I}{\rm Ind}_{H_{i}}^{G}\chi_{i}.$$ 
If $G$ is connected, the only open subgroup of $G$ is $G$ itself and a character on $G$ is $C_{0}$ if and only if $G$ is compact. This gives item 1. in the following result.

\begin{thm}\label{thm-lp}
Let $1\leq p<\infty$, $p\neq2$.
\begin{enumerate}
\item Let $G$ be a connected second countable locally compact group. Then $G$ has property $(H_{l_{p}})$ if and only if $G$ is compact.
\item Let $G$ be a totally disconnected locally compact second countable group. The following properties are equivalent :\\
(i) $G$ has property $(H_{l_{p}})$.\\
(ii) $G$ is amenable.
\end{enumerate}
\end{thm}

\begin{proof} We only have to show item $2$.\\
(i) $\Rightarrow$ (ii) : Let $\pi^{p}:G\rightarrow O(l_{p})$ be an orthogonal representation with almost invariant vectors, and with vanishing coefficients. Then we have $\pi^{2}\simeq\oplus_{i\in I}{\rm Ind}_{H_{i}}^{G}\chi_{i}$ for some open subgroups $H_{i}$ and unitary characters $\chi_{i}$ on $H_{i}$.\\
Since $\pi^{2}$ has the same form as $\pi^{p}$, $\pi^{2}$ has vanishing coefficients, and so does $\pi_{i}={\rm Ind}_{H_{i}}^{G}\chi_{i}$ for every $i\in I$. Let $i\in I$ be fixed. Then $\pi^{i}_{/H_{i}}$ contains $\chi_{i}$; indeed, we have 
$$\pi_{i}(h)\delta_{H_{i}}=\chi_{i}(h)\delta_{H_{i}}\textrm{ for all }h\in H_{i}.$$
It follows that $\chi_{i}\in C_{0}(H_{i})$ and hence $H_{i}$ is compact.\\
Since $H_{i}$ is compact, $\chi_{i}\subset\lambda_{H_{i}}$ and hence $\pi_{i}={\rm Ind}_{H_{i}}^{G}\chi_{i}\subset{\rm Ind}_{H_{i}}^{G}\lambda_{H_{i}}\subset\lambda_{G}$. So, $\pi^{2}$ is weakly contained in $\lambda_{G}$. Since we have also $1_{G}\prec\pi^{2}$, it follows that $1_{G}\prec\lambda_{G}$. By Hulanicki-Reiter theorem, $G$ is amenable.\\

(ii) $\Rightarrow$ (i) : Since $G$ is totally disconnected, by van Dantzig's theorem (see Theorem 7.7 in \cite{HewittRoss})), there exists a compact open subgroup $K$ of $G$. Let $(\pi^{2},l_{2}(G/K))$ be the quasi-regular representation of $G$ on $l_{2}(G/K)$, and let $\pi^{p}=M_{2,p}\circ \pi^{2}\circ M_{2,p}$. Notice that, for $g\in G$, every $\pi^{p}(g)$ is given by the same formula as $\pi^{2}(g)$ on the common dense subspace $l_{1}(G/K)$, so $\pi^{p}:G\rightarrow O(l_{p}(G/K))$ defines an orthogonal representation. Since $K$ is compact, $\lambda_{G/K}$ is contained in the regular representation $\lambda_{G}$ (by identifying $l_{2}(G/K)$ with the $K$-invariant functions in $L_{2}(G)$). Hence $\lambda_{G/K}$ has vanishing coefficients, since all matrix coefficients of $\lambda_{G}$ are $C_{0}$. It follows that $\pi^{p}$ has vanishing coefficients.\\
Since $G$ is amenable, the action of $G$ on $G/K$ is amenable (see p.28 in \cite{Eymard1972moyenne}). Hence, $\lambda_{G/K}$ almost has invariant vectors. Therefore, $G$ has $(H_{l_{p}})$.\\
\end{proof}

\section{Strongly mixing actions on $L_{p}(\mathcal{M})$}

In this section, we assume that $\mathcal{M}$ is a finite von Neumann algebra equipped with a normalized trace $\tau$. The main technical tools for the proof of Theorem \ref{thm-Hpmix} are the two following lemmas. 
\begin{lem}\label{lem-smJordan}
Let $G$ be a locally compact group. Let $\mathcal{M}$ be a finite von Neumann algebra, and let $1\leq p<\infty$, $p\neq2$. Let $\pi:G\rightarrow O(L_{p}(\mathcal{M}))$ be a strongly mixing orthogonal representation. Then $\pi(g)$ is a Jordan {*}-automorphism of $\mathcal{M}$ for every $g\in G$.
\end{lem}
\begin{proof}

Let $g\in G$. By Theorem \ref{thm2.3}, $\pi(g)$ has a Yeadon decomposition 
$$\pi(g)=u_{g}B_{g}J_{g}$$
with $u_{g}$ a unitary in $\mathcal{M}$, $B_{g}$ a positive operator affiliated with $\mathcal{M}$ such that its spectral projections commute with $\mathcal{M}$, and $J_{g}$ a {*}-Jordan isomorphism. Set $v_{g}=u_{g}B_{g}$ for all $g\in G$.\\

We will show that $u_{g}=1$ and $B_{g}=1$ for all $g\in G$. We claim that it suffices to give the proof when $p>2$. Indeed, if $p<2$, let $\pi^{p}(g)=u_{g}B_{g}J_{g}$ for every $g\in G$. By Proposition \ref{prp-contragradient}, the contragradient representation $(\pi^{p}){'}$ of $\pi^{p}$ on $L_{p'}$, with $p'>2$ the conjugate exponent of $p$, is given by the following formula
$$(\pi^{p}){'}(g)x=u_{g}^{*}B_{g}^{\frac{p}{p'}}u_{g}J_{g}(x)u_{g}^{*}\textrm{ for all }g\in G, x\in \mathcal{M}.$$
Moreover, the contragradient is obviously strongly mixing. Hence, if the claim is true for $p'>2$, then $u_{g}^{*}=1=u_{g}$, $B_{g}^{\frac{p'}{p}}=1=B_{g}$ and the claim is true for $p$.\\

So we can assume that $p>2$. For $g\in G$, $J_{g}(1)=1$ since $J_{g}$ is a Jordan isomorphism of $\mathcal{M}$. Since $\pi$ is strongly mixing, for $x=1$, we have 
$$\lim_{g\rightarrow\infty}\tau(\pi(g)(y))=\lim_{g\rightarrow\infty}\tau(\pi(g)(y)1)=\tau(y)\textrm{ for all }y\in\mathcal{M}.$$
Therefore, for $y=1$, we obtain 
$$\lim_{g\rightarrow\infty}\tau(v_{g})=1.$$
On the other hand, for $g_{0}\in G$ be fixed, we have
$$\lim_{g\rightarrow\infty}\tau(v_{gg_{0}})=\lim_{g\rightarrow\infty}\tau(\pi(g)\pi(g_{0})(1))=\tau(\pi(g_{0})(1))=\tau(v_{g_{0}}).$$
Hence $\tau(v_{g})=1$ for all $g\in G$. \\

Let $g\in G$ be fixed. Since $\pi(g)\in O(L_{p}(\mathcal{M}))$ and $1\in L_{p}(\mathcal{M})$, we have $\tau(\vert\pi(g)1\vert^{p})=\n 1\n_{p}^{p}=1$, that is $\tau(B_{g}^{p})=1$. Using twice H\"{o}lder's inequality, we have
$$1=\tau(u_{g}B_{g})\leq\tau(B_{g})\leq\tau(B_{g}^{t})^{1/t}\leq\tau(B_{g}^{p})^{1/p}=1\textrm{ for }1\leq t\leq p,$$
and it follows that $\tau(B_{g}^{2})=1$. Now by the Cauchy-Schwarz inequality,
$$1=\tau(v_{g})\leq \sqrt{\tau(B_{g}^{2})}=1.$$
The equality case gives that $v_{g}=u_{g}B_{g}=1$. From the uniqueness in the polar decomposition, it follows that $u_{g}=1$ and $B_{g}=1$. Hence the lemma is proved.
\end{proof}

The following corollary is a straightforward consequence of the previous Lemma \ref{lem-smJordan}, using Theorem \ref{thm2.9} in the case where $u_{g}=1$ and $B_{g}=1$ for all $g\in G$.
\begin{cor}\label{cor-smconjugate}
Let $1\leq p<\infty$, $p\neq2$, and $1\leq q<\infty$. Let $\pi^{p}:G\rightarrow O(L_{p}(\mathcal{M}))$ be a strongly mixing representation. Then the conjugate representation $\pi^{q}$ is strongly mixing.
\end{cor}

The following lemma asserts that the multiples of the unit $1\in \mathcal{M}$ are the only invariant vectors for a strongly mixing representation. We set 
$$L_{p}^{0}(\mathcal{M})=\{\esp x\in L_{p}(\mathcal{M})\esp\vert\esp\tau(x)=0\esp\}.$$
\begin{lem}\label{lem-smL0}
Let $G$ be a locally compact group. Let $\mathcal{M}$ be a finite von Neumann algebra and let $1\leq p<\infty$. Let $\pi^{p}:G\rightarrow O(L_{p}(\mathcal{M}))$ be a strongly mixing representation. Then $L_{p}(\mathcal{M}){'}(\pi^{p})=L_{p}^{0}(\mathcal{M})$. 
\end{lem}
\begin{proof}
Let $(\pi^{p}){'}:G\rightarrow O(L_{p'}(\mathcal{M}))$ be the contragradient representation of $\pi^{p}:G\rightarrow O(L_{p}(\mathcal{M}))$. Let $x\in L_{p}^{0}(\mathcal{M})$, and $y\in L_{p'}(\mathcal{M})^{(\pi^{p})'(G)}$. Then,
$$\tau(yx)=\lim_{g\rightarrow\infty}((\pi^{p})'(g)(y)x)=\lim_{g\rightarrow\infty}(y\pi^{p}(g^{-1})(x))=\tau(y)\tau(x)=0.$$
Thus $L_{p}^{0}(\mathcal{M})\subset L_{p}(\mathcal{M}){'}(\pi^{p})$.\\

To show that $L_{p}(\mathcal{M}){'}(\pi^{p})\subset L_{p}^{0}(\mathcal{M})$, it suffices to show that $1\in L_{p'}(\mathcal{M})^{(\pi^{p})'(G)}$. Indeed, if $1\in L_{p'}(\mathcal{M})^{(\pi^{p})'(G)}$, then for every $x\in L_{p}(\mathcal{M}){'}(\pi^{p})$, we have
$$\tau(x)=\tau(x1)=0.$$
Now let $g\in G$. By Lemma \ref{lem-smJordan}, $(\pi^{p})'(g)$ is a Jordan {*}-automorphism, so that $(\pi^{p})'(g)(1)=1$.
\end{proof}

Now we give the proof of Theorem \ref{thm-Hpmix}.

\begin{proof}[Proof of Theorem \ref{thm-Hpmix}]
$1.$ Let $\pi^{p}$ be a strongly mixing representation of $G$ on $L_{p}(\mathcal{M})$ which almost has invariant vectors in $L_{p}(\mathcal{M}){'}(\pi^{p})$. Then the conjugate representation $\pi^{2}$ defines a strongly mixing representation on $L_{2}(\mathcal{M})$ by Corollary \ref{cor-smconjugate}.\\

By Proposition \ref{prp-conjugatepi}, $\pi^{2}$ almost has invariant vectors in $L_{2}(\mathcal{M}){'}(\pi^{2})$. By Lemma \ref{lem-smL0}, we have $L_{2}(\mathcal{M}){'}(\pi^{2})=L_{2}^{0}(\mathcal{M})$. Since $\pi^{2}$ is strongly mixing, the restriction $\pi^{2}_{/L_{2}(\mathcal{M}){'}(\pi^{2})}$ of $\pi^{2}$ to $L_{2}(\mathcal{M}){'}(\pi^{2})$ almost has invariant vectors and has vanishing coefficients. Hence $G$ has property $(H)$. \\

$2.$ (ii) By Theorem 2.1.5 of Jolissaint in \cite{Cherix2001}, there exists an action $\pi$ of $G$ by automorphisms on the hyperfinite ${\rm II}_{1}$ factor $R$ such that \\
- $\pi$ is strongly mixing,\\
- there exists a non-trivial asymptotically invariant sequence $(e_{n})$ of projections in $R$, that is there exists a sequence of projections $(e_{n})$ such that $\tau(e_{n})=1/2$ for all $n$ and 
$$\lim_{n\rightarrow\infty}\sup_{g\in K}\vert\vert \pi(g)(e_{n})-e_{n}\vert\vert_{2}=0\textrm{ for all compact }K\textrm{ of }G.$$ 
This action defines a unitary representation $\pi^{2}:G\rightarrow \mathcal{U}(L_{2}(R))$ by 
$$\pi^{2}(g)x=\pi(g)(x)\textrm{ for all }g\in G,\esp x\in R,$$
as well as an orthogonal representation $\pi^{p}:G\rightarrow O(L_{p}(R))$ for every $1\leq p<\infty$. It is clear that $\pi^{p}$ is strongly mixing in the sense of Definition \ref{df-sm}.\\

By Lemma \ref{lem-smL0}, $L_{2}^{0}(R)=L_{2}(R){'}(\pi^{2})$. Define a sequence $(v_{n})$ in $R$ by  
$$v_{n}=e_{n}-\tau(e_{n})1.$$
Then $(v_{n})$ is a sequence of almost invariant vectors for $\pi$ in $L_{2}^{0}(R)=L_{2}(R){'}(\pi^{2})$. It is straightforward to check that $\vert\vert v_{n}\vert\vert_{2}^{2}=1/4$. Hence, $\pi^{2}$ has almost invariant vectors in $L_{2}(R){'}(\pi^{2})$. By Proposition \ref{prp-conjugatepi}, $\pi^{p}$ has almost invariant vectors in $L_{p}(R){'}(\pi^{p})$.\\

(i) The proof is similar as the previous proof, using Theorem 2.1.3 in \cite{Cherix2001} based on a construction due to Connes and weiss, instead of using Jolissaint's theorem. 
\end{proof}

\section{Affine actions by isometries on $L_{p}(\mathcal{M})$}

Here is a non-commutative analog of the implication $(i)\Rightarrow (ii)$ in Theorem \ref{thm-Nowak}.
\begin{thm}\label{thm-proper1}
Let $G$ be a second countable locally compact group with the Haagerup property. Let $1\leq p<\infty$. Then there exists a proper isometric affine action of $G$ on $L_{p}(l^{\infty}\otimes R)$, where $R$ is the hyperfinite ${\rm II}_{1}$ factor.
\end{thm}

\begin{proof}
We adapt the proof given in \cite{Nowak2006} to our non-commutative setting.\\

Take the orthogonal representation $\pi^{2}:G\rightarrow O(L_{2}(R))$ and the sequence $(v_{n})$ of almost invariant vectors for $\pi^{2}$ as in the proof of Theorem \ref{thm-Hpmix}. Set $w_{n}=M_{p,2}(v_{n})$ for all $n\in{\mathbb{N}}$, the images by the Mazur map $M_{p,2}$ of the vectors $v_{n}$. Notice that $\vert\vert w_{n}\vert\vert_{p}^{p}=\vert\vert v_{n}\vert\vert_{2}^{2}=\frac{1}{4}$. Using the uniform continuity of the Mazur map, one readily checks that $\pi^{p}=M_{2,p}\circ\pi^{2}\circ M_{p,2}$ defines an orthogonal representation on $L_{p}(R)$ such that
 $$\exists C>0\esp :\esp\lim_{g\rightarrow\infty}\vert\vert\pi^{p}(g)w_{n}-w_{n}\vert\vert_{p}^{p}\geq C\textrm{ for all }n\in\mathbb{N}.\esp(1)$$
 and
 $$\lim_{n\rightarrow\infty}\sup_{g\in K}\vert\vert\pi^{p}(g)w_{n}-w_{n}\vert\vert_{p}=0\textrm{ for all compact subset }K\textrm{ of }G.\esp(2)$$
 It follows from $(2)$ that the formula
 $$b(g)=\oplus_{n}\pi^{p}(g)w_{n}-w_{n}$$
 defines an element in $\oplus_{n}^{p}L_{p}(R)$. It is clear that $b:G\rightarrow \oplus_{n}^{p}L_{p}(R)$ is a $1$-cocycle associated to the representation $\oplus_{n}\pi^{p}$ to $\oplus_{n}^{p}L_{p}(R)$. By $(1)$, this cocycle is proper and the theorem is proved once we have noticed that $\oplus_{n}^{p}L_{p}(R)$ is isometrically isomorphic to $L_{p}(l^{\infty}\otimes R)$.
 \end{proof}
\begin{rem}{\rm
We don't know whether the converse of Theorem \ref{thm-proper1} is true for $p\leq2$. The method used in \cite{Nowak2006} in the classical $L_{p}$ breaks down in the non-commutative setting, since the distance associated to the norm of a non-commutative $L_{p}$-space (as $C_{p}$) is no longer a kernel conditionally of negative type. }
\end{rem}

We show in the following proof that from a proper cocycle with values in $L_{p}(l^{\infty}\otimes R)$, one can construct a proper cocycle with values in the space $L_{p}(\mathcal{B}(l_{2})\otimes R)$. Then Theorem \ref{thm-proper2} is a straightforward consequence of Theorem \ref{thm-proper1}.

\begin{proof}[Proof of Theorem \ref{thm-proper2}]

Take $b:G\rightarrow \oplus_{n}^{p}L_{p}(R)$ the proper $1$-cocycle constructed in the proof of Theorem \ref{thm-proper1}. Recall that $b$ is associated to a representation $\oplus_{n}\pi^{p}:G\rightarrow O(\oplus_{n}^{p}L_{p}(R))$, where $\pi^{p}:G\rightarrow O(L_{p}(R))$ is an orthogonal representation induced by an action by automorphisms of $R$, that is
$$\pi^{p}(g)x=\pi(g)x\textrm{ for all }g\in G, x\in R,$$
where every $\pi(g)$ is an automorphism of the algebra $R$. The isometric isomorphism between $\oplus_{n}^{p}L_{p}(R)$ and $L_{p}(l^{\infty}\otimes R)$ and the proper cocycle $b$ induce a proper $1$-cocycle $\overline{b}:G\rightarrow L_{p}(l^{\infty}\otimes R)$ associated to the orthogonal representation $\overline{\pi^{p}}:G\rightarrow O(L_{p}(l^{\infty}\otimes R))$ defined by
$$\overline{\pi^{p}}(g)(x_{n})_{n}\otimes a=(x_{n})_{n}\otimes\pi(g)a\esp\textrm{ for all }g\in G, (x_{n})_{n}\in l^{\infty}, a\in R. \esp (1)$$
Let us define 
$$\widetilde{\pi^{p}}(g)A\otimes a=A\otimes\pi(g)a\esp\textrm{ for all }g\in G, A\in \mathcal{B}(l_{2}), a\in R. \esp (2)$$
Since $\pi:G\rightarrow Aut(R)$ is a homomorphism from $G$ into the group of automorphisms of $R$, it is clear that $\widetilde{\pi^{p}}:G\rightarrow O(L_{p}(\mathcal{B}(l_{2})\otimes R))$ is a well-defined orthogonal representation of $G$ on $L_{p}(\mathcal{B}(l_{2})\otimes R)$. \\
Sequences in $l^{\infty}$ can be seen as multiplication operators in $\mathcal{B}(l_{2})$. This induces a linear and isometric embedding 
\begin{displaymath}
\begin{split}
&L_{p}(l^{\infty}\otimes R)\rightarrow L_{p}(\mathcal{B}(l_{2})\otimes R)\\
&\esp\esp\esp\esp\esp\esp\esp\esp\esp\esp\esp x\esp\esp\mapsto\esp\tilde{x}.
\end{split}
\end{displaymath}
By formulas $(1)$ and $(2)$, we have
$$\widetilde{\pi^{p}}(g)\tilde{x}=\widetilde{\overline{\pi^{p}}(g)x}\textrm{ for all }g\in G, x\in L_{p}(l^{\infty}\otimes R).$$
Now define $\tilde{b}(g)=\widetilde{\overline{b}(g)}$ for all $g\in G$. By the previous formula, $\tilde{b}:G\rightarrow L_{p}(\mathcal{B}(l_{2})\otimes R)$ is a $1$-cocycle associated to the representation $\tilde{\pi^{p}}$. Moreover, we have $\n \tilde{b}(g)\n_{p}=\n \overline{b}(g)\n_{p}$ for all $g\in G$. Hence the cocycle $\tilde{b}$ is proper, and Theorem \ref{thm-proper2} is proved.
\end{proof}

\bibliographystyle{amsplain}
\bibliography{biblio}

\end{document}